\documentclass[12pt]{amsart}

\title{}
\author{}

\usepackage[all]{xy}
\usepackage{fullpage, amsfonts, amsmath, amscd, amsthm,stmaryrd}
\usepackage{amssymb}
\usepackage{graphics}
\usepackage{graphicx}
\usepackage{amscd}

\newtheorem{theorem}{Theorem}[section]
\newtheorem{lemma}[theorem]{Lemma}
\newtheorem{definition}[theorem]{Definition}
\newtheorem{corollary}[theorem]{Corollary}
\newtheorem{proposition}[theorem]{Proposition}
\newtheorem{conjecture}[theorem]{Conjecture}
\newtheorem*{theorem*}{Theorem}

\theoremstyle{definition}

\newcommand{\CF}{\widehat{CF}}
\newcommand{\HF}{\widehat{HF}}
\newcommand{\HFK}{\widehat{HFK}}
\newcommand{\on}{\operatorname}

\newcommand{\F}{\mathcal{F}}

\newcommand{\C}{\mathbb{C}}
\newcommand{\X}{\mathbb{X}}
\newcommand{\B}{\mathbb{B}}
\newcommand{\A}{\mathbb{A}}
\newcommand{\Z}{\mathbb{Z}}

\newcommand\goth[1]{\mathfrak{#1}}
\newcommand{\s}{\goth{s}}
\newcommand{\fraks}{\mathfrak{s}}

\newcommand{\rk}{\on{rk}}

\newcommand{\sig}{\text{sig}\,}

\newcommand{\spinc}{${\rm Spin}^c$ }

\title{Cosmetic surgery in integral homology L-spaces}
\author{Zhongtao Wu}

\begin{document}

\maketitle

\begin{abstract}
Let $K$ be a non-trivial knot in $S^3$, and let $r$ and $r'$ be two distinct rational numbers of same sign, allowing $r$ to be infinite; we prove that there is no orientation-preserving homeomorphism between the manifolds $S^3_r(K)$ and $S^3_{r'}(K)$.
We further generalize this uniqueness result to knots in arbitrary integral homology L-spaces.     

\end{abstract}

\section{Introduction}

It has been known for a long time that every closed connected orientable three-manifold is obtained by surgery on a link in $S^3$. 
However, a classification of such three-manifold in terms of this surgery construction has remained elusive, due primarily to the lack of uniqueness of the surgery description.  

In particular, let $K$ be a framed knot in a closed oriented three-manifold $Y$. For a rational number $r$, let $Y_r(K)$ be the manifold obtained by Dehn surgery along $K$ with slope $r$.  Two surgeries along $K$ with distinct slopes $r$ and $r'$ are called \textit{equivalent} if there exists an orientation preserving homeomorphism of the complement of $K$ taking one slope to the other;  and they are called \textit{truly cosmetic} if there exists an orientation preserving homeomorphism between $Y_r(K)$ and $Y_{r'}(K)$.


When $K=U$ is the unknot in $S^3$, there are truly cosmetic surgeries:  $S^3_{p/q}(U)\cong S^3_{p/p+q}(U)$, and $S^3_{p/q_1} \cong S^3_{p/q_2}(U)$ when $q_1q_2\equiv 1 \,(\text{mod} \, p).$  While $p/q$ and $p/(p+q)$ are equivalent slopes, $p/q_1$ and $p/q_2$ as above are usually not.  By contrast, there are no known truly cosmetic surgeries on a nontrivial knot.  Indeed, it is one of the outstanding problems conjectured in Kirby's problem list, Problem 1.81(1):

\begin{conjecture}[Cosmetic Surgery Conjecture \cite{Kirby}]
Two surgeries on a nontrivial knot with nonequivalent slopes are never truly cosmetic.
\end{conjecture}

  
Lackenby \cite{L} has showed that under general conditions on a null-homotopic knot $K\subset Y$, there are at most finitely many truly cosmetic surgeries:  Consider two surgered manifolds $Y_{p/q}(K)$ and $Y'_{p'/q'}(K')$.  If $Y$ and $Y'$ are distinct, or $K$ and $K'$ are distinct, or $p/q\neq p'/q'$, then $Y_{p/q}(K) \ncong Y'_{p'/q'}(K')$ for a sufficiently large $|q|$.  

Techniques from Heegaard Floer homology were first introduced into the study of cosmetic surgery by Ozsv\'ath and Szab\'o in \cite {OSzRatSurg}.   Combining number theoretical methods, Wang was able to prove the conjecture for genus one knots in $S^3$ \cite{W}.  Ni, in another direction, used a twisted version of the Heegaard Floer homology to prove the case when K is a null-homologous knot in a closed three-manifold $Y$ that contains a non-separating sphere \cite{N}.  

Inspired by the work of Boyer and Lines \cite{BL}, we are able to prove the following uniqueness result. 

\begin{theorem} \label{main}
Let $r$ and $r'$ be two distinct rational numbers with $rr'>0$, and let $K$ be a non-trivial knot in $S^3$; then $S^3_r(K) \ncong S^3_{r'}(K)$.  
 
\end{theorem}

One immediate implication of this theorem is that there are at most TWO surgeries on a given non-trivial knot that can yield orientation preserving homeomorphic manifolds - no result of this type was previously known.  One should also compare it with Gordon and Luecke's \textit{knot complement} theorem \cite{GL}, which claims that $S_r^3(K)\ncong S^3= S^3_\infty(K).$ 


An \textit{integral homology $L$-space} is a three-manifold $Y$ such that $\HF(Y) \cong \Z$.  This is a class of three-manifolds with the simplest $\HF$, whose members include $S^3$, the Poincar\'e sphere $\Sigma(2,3,5)$, the Poincar\'e sphere with opposite orientation $-\Sigma(2,3,5)$, and their connected sums.  Theorem \ref{main} can be in fact generalized to knots in arbitrary integral homology $L$-spaces. 

 \begin{theorem} \label{L}
Let $r$ and $r'$ be two distinct rational numbers with $rr'>0$, and let $K$ be a non-trivial knot in an integral homology $L$-space $Y$; then $Y_r(K) \ncong Y_{r'}(K)$.\\  
 
\end{theorem}


Our proof of the above theorems involves a simultaneous application of \textit{Casson-Walker invariant}, \textit{Casson-Gordon invariant} and Heegaard Floer homology.  This argument can be potentially generalized to knots in Lens spaces and rational homology $L$-spaces, and it is still a work in progress.  Various other partial results and improvements may also be obtained, if we impose additional conditions on the knot $K$, such as its genus, or on the surgery coefficients.  

The paper is organized as follows.  In Section 2, we sketch some backgrounds in Casson-Walker invariant and Casson-Gordon invariant, and present explicit surgery formulae for the two invariants respectively.  In Section 3, we recall some preliminaries in Heegaard Floer homology, and make a few slight generalization on the setup, mostly from $S^3$ to integral homology $L$-space.  And finally in Section 4, we put everything together and prove our main results.     

\subsection*{Acknowledgment.}
I am grateful to my advisor, Zolt\'an  Szab\'o, for suggesting the problem and spending countless time in guiding and advising.  Thanks in addition to Steven Boyer, David Gabai, Joshua Greene, Yi Ni and JiaJun Wang for many helpful discussions at various points.  Finally, I would like to acknowledge Sam Lewallen for sharing the reference \cite{BL}.

\section{Preliminaries in classical three-manifold invariants}
We sketch in this section the definition and basic properties of two classical three-manifold invariants: the Casson-Walker invariant and the Casson-Gordon invariant. 

\subsection{Casson-Walker Invariant}

The Casson invariant is one of the many invariants of a closed three-manifold $Y$ that can be obtained by studying representations of its fundamental group in a certain non-abelian group $G$.  In particular, the Casson invariant of an integral homology sphere $Y$ can be obtained by counting representations of $\pi_1(Y)$ in $G=SU(2)$.  The geometric structures used to obtain a topological invariant is a Heegaard splitting of $Y$ and the symplectic geometry associated with it.  An alternative gauge-theoretical approach uses flat bundles together with a Riemannian metric on $Y$ and leads to a refinement of the Casson invariant, the Floer homology.  

Casson's $SU(2)$ intersection theory was later extended by Walker to include reducible representations, who generalized the invariant to rational homology spheres.  Most remarkably, Walker's invariant admits a purely combinatorial definition in terms of surgery presentations.  The existence and uniqueness of this invariant together with basic properties are given by the following theorem in Walker \cite{Walker}.           
\begin{theorem}
 
There exists a unique invariant $\lambda$, which satisfies the following properties:
\begin{enumerate}
 \item $\lambda$ coincides with Casson's invariant on integral homology sphere.
\item $\lambda (-Y)=-\lambda(Y)$ where $-Y$ stands for $Y$ with reversed orientation.
\item $\lambda(Y_1\#Y_2)=\lambda(Y_1)+\lambda(Y_2)$ for two rational homology spheres $Y_1$ and $Y_2$.
\item The number $12 \cdot |H_1(Y,\Z)|\cdot \lambda(Y)$ is an integer for any rational homology homology sphere.  
\item Let $k$ be a knot in a rational homology sphere $Y$, K its exterior, and $l\in \partial K$ a longitude.  Then, $\lambda$ satisfies the surgery formula
$$\lambda(K_a)=\lambda(K_b)+\tau(a,b;l)+\frac {\langle a,b \rangle}{\langle a,l \rangle\langle b,l \rangle} \cdot \Delta''_K(1).$$
for primitive $a,b \in H_1(\partial K, \Z)$ such that $\langle a,l \rangle \neq 0$ and $\langle b,l \rangle \neq 0$

\end{enumerate}

\end{theorem}

In the surgery formula above, The brackets $\langle \, , \rangle$ denote the intersection pairing $H_1(\partial K,\Z) \otimes H_1(\partial K, \Z) \longrightarrow \Z$, and $\Delta_k''$ stands for the second order derivative of the normalized Alexander polynomial of $k$.  Given a longitude $l$, choose a basis $x,y$ of $H_1(\partial K, \Z)$ such that $\langle x, y \rangle=1$ and $l=dy$ for some $d\in \Z$.  Then
$$\tau(a,b;l):=-s(\langle x,a \rangle, \langle y, a \rangle)+s(\langle x,b \rangle, \langle y, b \rangle)+ \frac {d^2-1}{12}\cdot \frac  {\langle a,b \rangle}{\langle a,l \rangle\langle b,l \rangle}$$
where $s(q,p)$ is the Dedekind sum defined by $$s(q,p):=\text{sign}(p)\cdot \sum^{|q|-1}_{k=1}((\frac{k}{p}))((\frac{kq}{p})),$$
where $$ ((x))=\begin{cases} x-[x]-\frac{1}{2}, & \text{if $x \notin \Z$}, \\
0, & \text{if $x \in \Z$},  \end{cases}$$

\bigskip
The surgery formula is much simplified when applied to null-homologous knots in rational homology spheres.

\begin{proposition} 
 Let $K$ be a null-homologous knot in a rational homology three-sphere $Y$, and let $L(p,q)$ be the lens space obtained by $(p/q)$-surgery on the unknot in $S^3$.  Then
$$\lambda(Y_{p/q}(K))=\lambda(Y)+\lambda(L(p,q))-\frac{q}{p}\Delta_K''(1).$$

\end{proposition}

Be aware that our definition of the Casson-Walker invariant may differ from that of various other references by a factor of 2.  Our normalization is made for the convenience of the following fact among other things.

\begin{proposition}
 For a lens space $L(p,q)$, $\lambda(L(p,q))=s(q,p)$. 
\end{proposition}

\subsection{Casson-Gordon Invariant}

In \cite{CG}, Casson and Gordon defined, for a closed oriented three-manifold $M$ and a surjective homomorphism $\phi: H_1(M) \longrightarrow \Z_m$, an invariant $\sigma_r(M,\phi)$, $0<r<m$.  The definition goes as follows.  Suppose $\tilde{M} \rightarrow M$ is the $m$-fold cyclic covering induced by $\phi$.  Pick up an $m$-fold cyclic branched covering of four-manifold $\tilde{W} \rightarrow W$, branched over a properly embedded surface $F$ in $W$, such that $\partial(\tilde{W}\rightarrow W)=(\tilde{M}\rightarrow M)$.  The existence of such $(W,F)$ follows from Lemma 2.2 of \cite{CG}.

The intersection form on $H_2(\tilde W)$ extends naturally to a nonsingular Hermitian from on $H:=H_2(\tilde{W}) \otimes \C$.  Let $\tau: H \longrightarrow H$ be the automorphism induced by the covering transformation of $\tilde{W}$.  Note that $\tau$ is an isometry of $(H,\cdot)$, and that $\tau^m=\text{id}$.  Write $\omega=e^{2\pi i/m}$, and let $E_r$ be the $\omega^r$-eigenspace of $\tau$, $0\leq r<m$.  Then $(H,\cdot)$ decomposes as an orthonormal direct sum $E_0\oplus E_1 \oplus \cdots \oplus E_{m-1}$.  

Let $\varepsilon_r(\tilde{W})$ be the signature of the restriction of $\cdot$ to $E_r$, and denote by $\sig(W)$ the signature of $W$.  Then, define for $0<r<m$, the rational number
$$\sigma_r(M,\phi)=\sig (W)-\varepsilon_r(\tilde{W})-\frac{2[F]^2r(m-r)}{m^2}.$$         

Novikov additivity and the $G$-signature theorem show that $\sigma_r(M, \phi)$ is independent of the choice of the cyclic branched cover $\tilde{W} \rightarrow W$, and that it depends only on the cyclic cover $\tilde{M}\rightarrow M$ and $r$.  Hence, when $H_1(M)=\Z_m$, the sum $\sum_{r=1}^{m-1}\sigma_r(M,\phi)$ is an invariant of the manifold M, independent of $\phi$.  

\begin{definition}
Define the total Casson-Gordon invariant of $M$ to be $$\tau(M)= m \, \text{sig}(W)-\text{sig}(\tilde{W})-\frac{[F]^2(m^2-1)}{3m}. $$
\end{definition}

Let $K$ be a knot in an integral homology sphere $Y$ and $m\neq 0$.  We set $$\sigma(K,m)=\sum_{r=1}^{m-1}\sigma_K(e^{2i\pi r/m}),$$ 
where $\sigma_K(\xi)$ is the signature of the matrix $A(\xi):=(1-\bar{\xi})A+(1-\xi)A^{T}$ for a Seifert matrix $A$ of $K$.  $|\xi|=1$.

A surgery formula for the total Casson-Gordon invariant was established in \cite{BL}.

\begin{proposition}
Let $K$ be a knot in an integral homology sphere $Y$, then $$\tau(Y_{p/q}(K))=\tau(L(p,q))-\sigma(K,p).$$
\end{proposition}

The value of the invariant for the Lens space $L(p,q)$ is given by: 
\begin{proposition}
For a lens space $L(p,q)$, $\tau(L(p,q))=-4p\cdot s(q,p)$.  

\end{proposition}

\section{Preliminaries in Heegaaard Floer homology}

Heegaard Floer homology is an invariant for closed three manifolds Y (see Ozsv\'ath-Szab\'o \cite {OSzAnn1}\cite{OSzAnn2}).  The invariant, denoted $HF^\circ(Y)$, is the homology of a chain complex whose generators have a combinatorial definition, and whose boundary operator counts certain pseudo-holomorphic disks in associated spaces.  

In  Ozsv\'ath-Szab\'o \cite{OSzKnot} and Rasmussen \cite{Ra}, a closely related invariant is defined for null-homologous knots $K$ in a closed, oriented three-manifold $Y$, taking the form of an induced filtration on the Heegaard Floer complex of $Y$.  The filtered chain homotopy type of this complex is a knot invariant, known as ``knot Floer homology''. Given an integer $n$, let $Y_n(K)$ denote the three-manifold obtained by $n$-framed surgery on $Y$ along $K$.  When $n$ is sufficiently large, there is an immediate relationship between the knot Floer homology of $K$ and the Heegaard Floer homology of $Y_n(K)$, see \cite [Section 4]{OSzKnot}.  The general case for an arbitrary integer or rational surgery on an integral homology sphere $Y$ is developed by Ozsv\'ath and Szab\'o in \cite {OSzIntSurg} \cite{OSzRatSurg}, which we review next.       

\subsection{Rational surgery formulas}

Knot Floer homology associates to $K$ a
$\Z\oplus\Z$--filtered $\Z[U]$-complex $C=CFK^\infty(Y,K)$, generated over $\Z$ by a set $X$ equipped with a function
$\F: X\longrightarrow\Z\oplus \Z$ with the property that, if $\F({\bf x})=(i,j)$,
then $\F(U\cdot{\bf x})=(i-1,j-1)$ and $\F ({\bf y})\leq\F({\bf x})$
for all $\bf y$ having nonzero coefficient in $\partial{\bf x}$.

For a region $S$ in the plane with the property that $(i,j)\in S$
implies $(i+1,j),(i,j+1)\in S$, let $C\{S\}$ be the natural quotient complex of $C$ generated by $\bf x$ with $\F({\bf x})\in S$. For an
integer $s$, we define $A^+_{\smash{s}}(K):=C\{\max(i,j-s)\geq0\}$ and $B^+(K):=C\{i\geq0\}$.  There are two canonical chain maps $v_s^+: A^+_s\longrightarrow B^+$ and $h_s^+: A^+_s\longrightarrow B^+$.  The map $v_s^+$ is projection onto $C\{i\geq 0\}$, while $h_s^+$ is projection onto  $C\{j\geq s\}$, followed by the identification with  $C\{j\geq 0\}$, followed by the chain homotopy equivalence from  $C\{j\geq 0\}$ to
 $C\{i\geq 0\}$.



Now fix a surgery slope $p/q$ and suppose $q>0$. Consider the two chain
complexes
$$\A^+=\bigoplus_{t\in\Z}(t,A^+_{\lfloor\frac{t}{q}\rfloor}),\quad
\B^+=\bigoplus_{t\in\Z}(t,B^+),$$ where $\lfloor x\rfloor$ is the greatest
integer not bigger than $x$. An element of $\A^+$ could be written as
$\{(t,a_t)\}_{t\in\Z}$ with $a_t\in A^+_{\lfloor\frac tq\rfloor}$. Define a
chain map $D^+_{\smash{p/q}}: \A^+\rightarrow\B^+$ by
$$D^+_{\smash{p/q}}\{(t,a_t)\}=\{(t,b_t)\},$$ 
where $$\quad b_t=v^+(a_t)+h^+(a_{t-p}).$$ Let $\X_{\smash{p/q}}^+(K)$ be the mapping cone of
$D^+_{\smash{p/q}}$. 

Note that $\X^+_{\smash{p/q}}$ naturally splits into the direction sum of $p$ subcomplexes
$$\X^+_{p/q}=\bigoplus_{i=0}^{p-1}\X^+_{i,p/q},$$
where $\X^+_{i,p/q}$ is the subcomplex of $\X^+_{\smash{p/q}}$ containing all
$A^+_t$ and $B^+_t$ with $t\equiv i\pmod{p}$.

The Heegaard Floer homology of a $p/q$
surgered manifold is determined by the mapping cone $\X^+_{\smash{p/q}}$ according to the following theorem of Ozsv{\'a}th and Szab{\'o}.

\begin{theorem}[Ozsv{\'a}th and Szab{\'o} \cite{OSzRatSurg}]
 Let $K \subset
Y$ be a nullhomologous knot and $p,q$ a pair of coprime integers.
Then, for each $i\in \Z/p\Z$, there is a relatively graded isomorphism of groups
$$ HF^+(Y_{\smash{p/q}}(K),\fraks_i) \cong H_*(\X^+_{i,p/q}),$$ 
where $\fraks_i$ is the \spinc structure
corresponding to $i\in\Z/p\Z$.
\end{theorem}

Restricting the preceding formulation to the ``boundary'', we obtain a similar surgery formula for $\HF$.  Let $\hat{A}_{\smash{s}}(K):=C_K\{\max(i,j-s)=0\}$ and $B^+(K):=C\{i=0\}$, and the maps $\hat{v}$, $\hat{h}$ be the induced projections accordingly; we form the mapping cone $ \hat{\X}_{i, \smash{p/q}}(K)$.  

\begin{theorem}[Ozsv{\'a}th and Szab{\'o} \cite{OSzRatSurg}]     \label{rationalsurgeryformula}
 Let $K \subset
Y$ be a nullhomologous knot and $p,q$ a pair of coprime integers.
Then, for each $i\in \Z/p\Z$, there is a relatively graded isomorphism of groups
$$ \HF(Y_{\smash{p/q}}(K),\fraks_i) \cong H_*(\hat{\X}_{i,p/q}),$$ 
where $\fraks_i$ is the \spinc structure
corresponding to $i\in\Z/p\Z$.
\end{theorem}

\subsection{A rank formula}
Before applying the rational surgery formula to find the rank of $\HF$ of a surgered manifold, we define several invariants pertaining to a chain complex $CFK^\infty$.  From now on, $Y$ is assumed to be an integral homology $L$-space, unless otherwise specified.  

Let $\F_Y(K,m) := C\{i=0 ,\, j\leq m \} \subset \CF(Y)$ be the subcomplex generated by intersection points whose filtration level is less than or equal to $m$.  We obtain a sequence of maps $$\iota^m_K: \F_Y(K,m)\longrightarrow \CF(Y),$$
which induce isomorphisms in homology for all sufficiently large integers $m$.  

\begin{definition}
Define $\tau_Y(K)$ by $$\tau_Y(K)=\min\{m\in \Z | \iota^m_K: \F_Y(K,m) \longrightarrow \CF(Y) \, \text{induces a non-trivial map in homology} \}$$

\end{definition}

The invariant $\tau_Y(K)$ is an invariant of the knot $K$, which gives a lower bound on the four-ball genus of the knot $K$ when $Y=S^3$.  Likewise, we can define an invariant $\nu_Y(K)$ in a similar manner.
  
\begin{definition}
 Define $\nu_Y(K)$ by $$\nu_Y(K)=\min\{s\in \Z | \hat{v}_s: \hat{A}_s \longrightarrow \CF(Y) \, \text{induces a non-trivial map in homology} \}.$$
\end{definition}

\bigskip
Let $m(K)\subset -Y$ be the mirror image of the knot $K \subset Y$; that is, so to speak, reversing the orientation of the ambient manifold $Y$ while keeping the original orientation on $K$.  When $Y=S^3$, $m(K)$ is the mirror image in the usual sense if we apply the orientation reversing homeomorphism of $S^3$, namely, the reflection.  In general, $Y\neq -Y$; but we claim the following identity 
$$\tau_Y(K)=-\tau_{-Y}(m(K)).$$

The proof is a direct adaption of \cite[Lemma 3.3]{OSzFourBall}, that follows from the duality map $\mathcal{D}: \CF_*(Y)\rightarrow \CF^*(-Y)$.   

We also claim that $$\nu_Y(K)=\tau_Y(K) \;\; \text{or} \;\; \tau_Y(K)+1.$$

When $Y=S^3$, this is proved in \cite[Proposition 3.1]{OSzFourBall} based on a purely homological-algebra argument; this argument evidently hold for the filtered chain group $CFK^\infty(Y,K)$ as well, because $S^3$ cannot be distinguished from an arbitrary integral homology $L$-space in the chain level.            
\bigskip

Therefore, either $\nu_Y(K)$ or $\nu_{-Y}(m(K))$ is non-negative.  
\bigskip

Now, we calculate the rank of $\HF(Y_{p/q}(K))$.  Since $\rk(Y_{p/q}(K))=\rk((-Y)_{-p/q}(m(K))$, it suffices to consider knots $K$ with $\nu_Y(K) \geq \nu_{-Y}(m(K))$.  Then, we have the following rank formula analogous to \cite[Proposition 9.6]{OSzRatSurg}.

\begin{proposition}\label{rank}
Let $K$ be a knot in an integral homology $L$-space $Y$, 
and fix a pair of relatively prime integers $p$ and $q$ with $p \neq 0$ and $q>0$; and suppose that $\nu_Y(K)\geq \nu_{-Y}(m(K))$.  Then, if $\nu_Y(K)>0$ or $p>0$, $$\rk(\HF(Y_{p/q}(K))=p+2\max(0, (2\nu_Y(K)-1)q-p)+q(\sum_s(\rk H_*(\hat{A}_s)-1));$$
if $\nu_Y(K)=0$, we have that  $$\rk(\HF(Y_{p/q}(K))=|p|+q(\sum_s(\rk H_*(\hat{A}_s)-1)).$$

\end{proposition}

\begin{lemma}
 
Suppose $K$ is a knot in an integral homology $L$-space $Y$, 
and suppose that $\nu_Y(K)=\nu_{-Y}(m(K))=0$.  Then, the image of 
$$(\hat{h}_0\oplus \hat{v}_0)_* \longrightarrow H_*(\hat{B}\oplus\hat{B})\cong \mathbb{F}\oplus \mathbb{F}$$
is one-dimensional.  
\end{lemma}

\begin{proof}
Let $\delta(K)$ denote the dimension of the image of  $$(\hat{h}_0\oplus \hat{v}_0)_* \longrightarrow H_*(\hat{B}\oplus\hat{B})\cong \mathbb{F}\oplus \mathbb{F};$$
and let $a_0(K)$ denote the dimension of $H_*(\hat{A}_0(K))$.  Note that both $\delta(K)$ and $\delta(m(K))$ are positive integers by the definition of $\nu$.  

For sufficiently large $N$, we have $$\rk \HF(Y_N(K),0)=a_0(K).$$ 
Using Theorem \ref{rationalsurgeryformula}, we have $$\rk \HF(Y_{-N}(K),0)=a_0(K)+2-2\delta(K).$$
Similar formulae hold for large $N$ and $-N$ surgeries on $m(K)\subset -Y$.  Then, from the identification $Y_N(K)=-(-Y)_{-N}(m(K))$, we conclude that 
$$a_0(K)=a_0(m(K))+2-2\delta(m(K))$$
$$a_0(m(K))=a_0(K)+2-2\delta(K).$$

from which we see $\delta(K)=\delta(m(K))=1$. 


\end{proof}
Express $ \HF(Y_{p/q}(K))$ in terms of $H_*(\hat{\X}_{p/q})$.  Propostion \ref{rank} now follows by exactly the same argument of \cite{OSzRatSurg}.

\subsection{$L$-space surgeries on knots}
A rational homology three-sphere $Y$ is called an $L$-space if $\HF(Y,\s) \cong \Z$ for each \spinc structure $\s$ on $Y$.  
A knot $K \subset S^3$ is said to admit an $L$-space surgery if certain rational surgery $S^3_r(K)$ is an $L$-space.  Such knot has a very special knot Floer homology \cite[Theorem 1.2]{OSzLensSpace}.    

When $K$ is a knot in an integral homology $L$-space $Y$, we can obtain a similar characterization of $\HFK(Y,K)$. 

\begin{proposition}
\label{thm:FloerHomology}
Suppose $K\subset Y$ is a knot in an integral homology $L$-space.  If there is a rational number $r$ for which $Y_r(K)$ is an $L$-space, then there is an
increasing sequence of integers
$n_{-k}<...<n_k$
with the property that $n_i=-n_{-i}$, 
and $\HFK(K,j)=0$ unless $j=n_i$ for some $i$, in which case
$\HFK(K,j)\cong \Z$.
\end{proposition}

Once again, observe that the original proof for $S^3$ utilizes only the information from the filtered chain.  So the entire argument can be carried through here without any change. 

For our purpose, the Alexander polynomials of such knots are particularly useful.

\begin{corollary}
\label{StructAlex}
Let $K$ be a knot that admits an $L$-space surgery. Then the Alexander polynomial of $K$ has the form
$$\Delta_K(T) = (-1)^k+ \sum_{j=1}^k(-1)^{k-j} (T^{n_j}+T^{-n_j}),$$
for some increasing sequence of positive integers $0<n_1<n_2<...<n_k$.
\end{corollary}

\section{Proof of the main results}

\begin{proof}[Proof of Theorem \ref{main} and \ref{L}]

Suppose there were truly cosmetic surgeries for a nontrivial knot $K$ in an integral homology sphere $Y$.  Since the surgeried manifold has $H_1(Y_{p/q}(K))=\Z/p\Z$, 
we can assume $Y_{p/q}(K)\cong Y_{p/q'}(K)$.  (The $\infty$ surgery can be treated as an $1/0$ surgery.)  Comparing their respective Casson-Walker invariants and Casson-Gordon invariants, resorting to the surgery formulae:
$$\lambda(Y)+\lambda(L(p,q))-\frac{q}{p}\Delta_K''(1)=\lambda(Y)+\lambda(L(p,q'))-\frac{q'}{p}\Delta_K''(1),$$
$$\tau(L(p,q))-\sigma(K,p)=\tau(L(p,q'))-\sigma(K,p),$$
and plugging the results for Lens spaces, we conclude that all manifolds $Y_{p/q}(K)$ are distinct provided $\Delta''_K(1) \neq 0$.  See also \cite[Proposition 5.1] {BL}. 

We shall see that the Heegaard Floer homology suppies exactly the missing piece.  Suppose $q$ and $q'$ have the same sign, then the rank formula in Proposition \ref{rank} implies $\rk(\hat{A}_s)=1$ for all $s$, and $\rk\HF(Y_{p/q}(K))=\rk\HF(Y_{p/q'}(K))=p$.  Hence, $K$ is a knot that admits an $L$-space surgery, and consequently $\Delta_K(T) = (-1)^k+ \sum_{j=1}^k(-1)^{k-j} (T^{n_j}+T^{-n_j})$  by Corollary \ref{StructAlex}.

We claim that $\Delta_K''(1)\neq 0$ unless $\Delta_K(T)=1$.  This follows from a straightforward computation $$\Delta_K''(1)=2\sum_{j=1}^{k} (-1)^{k-j} n_j^2.$$  
and the fact that $0<n_1<n_2<...<n_k$.

For those knots $K$ with $\Delta_K(T)=1$, Proposition \ref{thm:FloerHomology} implies, in addition, that $\HFK(K,0)=\Z$, and $\HFK(K,j)=0$ for any other $j$.  Hence $K=U$, by the fact that knot Floer homology detects the unknot \cite{Ni}.  This finishes the proof of the theorems.   

\end{proof}



\end{document}